\documentclass[a4paper,11pt]{amsart}

\usepackage{mathpazo}
\usepackage{amssymb}
\usepackage[pagebackref,
    ,pdfborder={0 0 0}
    ,urlcolor=black,a4paper,hypertexnames=false]{hyperref}
\hypersetup{pdfauthor=Clara L\"oh,pdftitle=Which finitely generated Abelian groups admit equal growth functions?}

\usepackage{pgf}
\usepackage{tikz}
\usetikzlibrary{decorations.pathmorphing}

\newtheorem{thm}{Theorem}[section]
\newtheorem{prop}[thm]{Proposition}
\newtheorem{lem}[thm]{Lemma}
\newtheorem{cor}[thm]{Corollary}

\theoremstyle{remark}
\newtheorem{rem}[thm]{Remark}
\newtheorem{exa}[thm]{Example}

\theoremstyle{definition}
\newtheorem{defi}[thm]{Definition}

\usepackage{amsfonts}
\newcommand{\Z}{\mathbb{Z}}
\newcommand{\Q}{\mathbb{Q}}
\newcommand{\R}{\mathbb{R}}
\newcommand{\N}{\mathbb{N}}

\DeclareMathOperator{\id}{id}
\DeclareMathOperator{\rk}{rk}

\DeclareMathOperator{\tors}{tors}
\DeclareMathOperator{\diam}{diam}
\def\epsilon{\varepsilon}

\def\gvertex#1{%
  \fill #1 circle(0.1);
}

\usepackage{color}
\usepackage{pdfcolmk}

\def\draftinfo{}

\author{Clara L\"oh}
\author{Matthias Mann}
\title[Finitely generated Abelian groups with equal growth functions]{Which finitely generated Abelian groups admit\\ equal growth functions?}
\date{\today.\ \copyright{\ C.~L\"oh, M.~Mann 2013}, \draftinfo
     MSC~2010 classification: 05C25, 05C63, 20F65}

\begin{document}

\begin{abstract}
  We show that finitely generated Abelian groups admit equal growth 
  functions with respect to symmetric generating sets if and only if 
  they have the same rank and the torsion parts have the same parity. 
  In contrast, finitely generated Abelian groups 
  admit equal growth functions with respect to monoid generating sets 
  if and only if they have same rank. 
  Moreover, we show that the size of the torsion part is in fact
  determined by the set of \emph{all} growth functions of a finitely
  generated Abelian group.
\end{abstract}

\maketitle

\section{Introduction}

Cayley graphs of finitely generated Abelian groups are rather rigid:
Isomorphisms between Cayley graphs of finitely generated Abelian
groups are almost affine~\cite[Theorem~1.3]{loeh} and hence finitely
generated Abelian groups admit isomorphic Cayley graphs if and only if
they have the same rank and the torsion parts have the same
size~\cite[Corollary~1.4]{loeh}. Moreover, it is known that the rank
coincides with the growth rate~\cite[Chapter~VI]{delaharpe} and 
also that the parity of the torsion part is encoded in the growth function
of any Cayley graph~\cite{mathoverflow, loeh}. Thus it is a natural
question whether also the exact size of the torsion part can be read
off the growth functions of finitely generated Abelian
groups~\cite[Problem~4.1]{loeh}.

In the following, we show that finitely generated Abelian groups (of
non-zero rank) admit equal growth functions with respect to
\emph{symmetric} generating sets if and only if they have the same rank and
the torsion parts have the same parity (Theorem~\ref{thmsymmetric}).
In contrast, finitely generated Abelian groups (of non-zero rank)
admit equal growth functions with respect to \emph{monoid} generating sets
if and only if they have same rank (Theorem~\ref{thmmonoid}). However,
we show that the size of the torsion part is in fact determined by the
set of \emph{all} growth functions of a finitely generated Abelian
group (Corollary~\ref{corgrowthset}).

We now describe the results in more detail. For the sake of
completeness, let us briefly recall some basic notation: A subset~$A
\subset G$ of a group~$G$ is called \emph{symmetric} if for all~$g \in
A$ also~$g^{-1} \in A$.


\begin{defi}[Word metric, (spherical) growth function]
  Let $G$ be a finitely generated group and let $S$ be a (not necessarily 
  symmetric) finite monoid generating set of~$G$.
  \begin{itemize}
    \item The \emph{word metric on~$G$ with respect to~$S$} is defined 
      by 
      \begin{align*}
        d_S \colon G \times G & \longrightarrow \R_{\geq 0} \\
        (g,h) & \longmapsto 
        \min\bigl\{ n \in \N
            \bigm| \exists_{\,s_1, \dots, s_n \in S} \;\;\; 
                   h^{-1} \cdot g = s_1 \cdot \dots \cdot s_n
            \bigr\}.
      \end{align*}
      (Notice that the ``metric''~$d_S$ in general will \emph{not} 
      be symmetric if~$S$ is not symmetric.)
    \item For~$r \in \N$ we write $B_{G,S}(r) := \{ g \in G \bigm|
      d_S(g,e) \leq r \bigr\}$ for the \emph{ball of radius~$r$}
      around the neutral element~$e$ in~$G$ with respect to the word
      metric~$d_S$.
    \item The \emph{spherical growth function of~$G$ with respect to~$S$} 
      is given by
      \begin{align*}
        \sigma_{G,S} \colon \N & \longrightarrow \N\\
        r & \longmapsto 
        \bigl| \bigl\{ g \in G \bigm| d_S(g,e) = r \bigr\}\bigr|
        = \bigl| B_{G,S} (r) \setminus B_{G,S}(r-1)\bigr|.
      \end{align*}
    \item The \emph{growth function of~$G$ with respect to~$S$} 
      is given by
      \begin{align*}
        \beta_{G,S} \colon \N & \longrightarrow \N\\
        r & \longmapsto 
        \bigl| B_{G,S}(r)\bigr|
        = \sum_{t=0}^r \sigma_{G,S}(t).
      \end{align*}
  \end{itemize}
\end{defi}

Via the \v Svarc-Milnor lemma, growth functions of 
groups are related to volume growth functions of Riemannian
manifolds~\cite[Theorem~IV.23, Prop\-osition~VI.36]{delaharpe}.
Furthermore, growth functions of groups contain valuable large-scale
geometric information that plays an important role in geometric group
theory~\cite[Chapters~VI--VIII]{delaharpe}.

\begin{defi}
  Two finitely generated groups~$G$ and~$G'$ \emph{admit equal growth 
  functions} if there exist finite monoid generating sets~$S
  \subset G$ and $S' \subset G'$ of~$G$ and~$G'$ respectively such
  that the corresponding growth functions coincide, i.e., such 
  that $\beta_{G,S} = \beta_{G',S'}$. Analogously, $G$ and $G'$ \emph{admit 
    equal growth functions with respect to symmetric generating sets} 
  if there exist finite symmetric generating sets~$S
  \subset G$ and $S' \subset G'$ of~$G$ and~$G'$ respectively with~$\beta_{G,S} = \beta_{G',S'}$.
\end{defi}

If $G$ is a finitely generated Abelian group, then the \emph{torsion
  subgroup~$\tors G$} of~$G$, i.e., the subgroup of all elements
of~$G$ of finite order, is a finite group. Moreover, the
quotient~$G/\tors G$ is a finitely generated free Abelian group and
the rank of~$G/\tors G$ is called the \emph{rank~$\rk G$ of~$G$}. In 
this situation, one has~$G \cong G/\tors G \times \tors G \cong 
\Z^{\rk G} \times \tors G$.

\begin{thm}[The case of symmetric generating systems]\label{thmsymmetric}
  Two finitely generated Abelian groups of non-zero rank admit equal
  growth functions with respect to symmetric generating sets if and
  only if they have the same rank and the torsion parts have the same
  parity.
\end{thm}

\begin{thm}[The case of monoid generating systems]\label{thmmonoid}
  Two finitely generated Abelian groups of non-zero rank admit 
  equal growth functions (with respect to monoid generating sets) if 
  and only if they have the same rank.
\end{thm}

Moreover, we observe that the size of the torsion part is a lower 
bound for the ratio between any growth function and the standard 
growth functions of the corresponding free part:

\begin{defi}[Standard growth functions]\label{defstandardgrowth}
  Let $d \in \N$, let $E_d \subset \Z^d$ be the standard basis, 
  and let $v_d := (-1,\dots,-1) \in \Z^d$. Then $E_d \cup (-E_d)$ 
  is a finite symmetric generating set of~$\Z^d$ and $E_d \cup \{v_d\}$ 
  is a finite monoid generating set of~$\Z^d$. We write 
  \begin{align*}
    \beta_d  := \beta_{\Z^d, E_d \cup (-E_d)}
    \quad\text{and}\quad
    \beta^+_d  := \beta_{\Z^d, E_d \cup \{v_d\}}.
  \end{align*}
\end{defi}

\begin{prop}[Minimal growth]\label{propmingrowth}
  Let $G$ be a finitely generated Abelian group.
  \begin{enumerate}
    \item If $S$ is a finite symmetric generating set of~$G$, then
      \[ \limsup_{r \rightarrow \infty} 
      \frac{\beta_{G,S}(r)}{\beta_{\rk G}(r)} 
      \geq \bigl|\tors(G)\bigr|.
      \]
    \item If $S$ is a finite monoid generating set of~$G$, then 
      \[ \limsup_{r \rightarrow \infty} 
      \frac{\beta_{G,S}(r)}{\beta^+_{\rk G}(r)} 
      \geq \bigl|\tors(G)\bigr|.
      \]
  \end{enumerate}
\end{prop}

As a consequence we obtain that only finitely many isomorphism types
of finitely generated Abelian groups can share a single growth
function and that the size of the torsion part can be recovered from
the set of \emph{all} growth functions of a finitely generated Abelian
group:

\begin{cor}[Finite ambiguity]\label{corfinite}
  Let $\beta \colon \N \longrightarrow \N$ be a function.
  Then there are at most finitely many 
  isomorphism types of finitely generated Abelian groups~$G$ that 
  have a finite monoid generating set~$S$ with~$\beta_{G,S} = \beta$.
\end{cor}

\begin{cor}[Recognising the size of the torsion part from the set 
  of growth functions]\label{corgrowthset}
  Let $G$ and $G'$ be finitely generated Abelian groups.
  \begin{enumerate}
  \item
    If the sets of all growth functions of $G$ and $G'$ with respect
    to symmetric generating sets coincide, i.e.,
      \begin{align*} 
         & \; \{ \beta_{G,S} 
              \mid \text{$S$ is a finite symmetric generating set of~$G$}\}
         \\
       = & \; \{ \beta_{G',S'} 
              \mid \text{$S'$ is a finite symmetric generating set of~$G'$}\},
      \end{align*}
      then $\rk G = \rk G'$ and $|\tors G| = |\tors G'|$.
  \item
    If the sets of all growth functions of $G$ and $G'$ coincide, i.e., 
      \begin{align*} 
         & \; \{ \beta_{G,S} 
              \mid \text{$S$ is a finite monoid generating set of~$G$}\}
         \\
       = & \; \{ \beta_{G',S'} 
              \mid \text{$S'$ is a finite monoid generating set of~$G'$}\},
      \end{align*}
      then $\rk G = \rk G'$ and $|\tors G| = |\tors G'|$.
  \end{enumerate}
\end{cor}

However, the converse of Corollary~\ref{corgrowthset} does \emph{not}
hold (Example~\ref{exaconverse}).

In fact, we will prove the above results for the
slightly larger class of groups of type~$\Z^d \times F$, where $d \in
\N$ and $F$ is a finite group.

This paper is organised as follows: In Section~\ref{secbasics}, we
recall some basics about growth functions. In Section~\ref{secnec},
we deduce that the conditions given in Theorem~\ref{thmsymmetric}
and~\ref{thmmonoid} are necessary; conversely, in
Section~\ref{secsuff}, we present examples that show that these
conditions are also sufficient. In Section~\ref{secmin}, we discuss
minimal growth of finitely generated groups and prove
Proposition~\ref{propmingrowth}, as well as its consequences
Corollary~\ref{corfinite} and Corollary~\ref{corgrowthset}.

\section{Preliminaries on growth functions}\label{secbasics}

For the sake of completeness, we collect some basic facts about growth
functions of groups~\cite[Chapter~VI--VIII]{delaharpe}, in particular,
of groups that are products of finitely generated free Abelian groups
and a finite group.

\begin{prop}[Changing the generating set]\label{propgenset}
  Let $G$ be a finitely generated group, and let $S, T \subset G$ be
  finite monoid generating sets of~$G$. Then there exists~$C \in
  \N_{>0}$ such that for all~$r\in \N$ we have
  \[ \beta_{G,T}(r) \leq \beta_{G,S}(C\cdot r)
     \quad
     \text{and}
     \quad
     \beta_{G,S}(r) \leq \beta_{G,T}(C \cdot r).
  \]
\end{prop}
\begin{proof}
  Because $S$ and $T$ are finite, the sets $\{ d_S(t,e) \mid t \in
  T\}$ and $\{ d_T(s,e) \mid s \in S\}$ are finite and so have finite
  upper bounds. Rewriting minimal length representations in one
  generating system in terms of the other one shows that there is a
  constant~$C \in \N_{>0}$ such that for all~$g \in G$ we have
  \[ d_S(g,e) \leq C \cdot d_T(g,e)
     \quad\text{and}\quad
     d_T(g,e) \leq C \cdot d_S(g,e),
  \]
  from which the claim follows.
\end{proof}

\begin{prop}[Polynomial growth rate and rank]\label{proprankgrowth}\label{propliminf}
  Let $d \in \N_{>0}$, let $F$ be a finite group, and let $G \cong
  \Z^d \times F$. Let $S \subset G$ be a finite monoid generating set
  of~$G$.
    Then for all~$r\in \N$ we have
    \[ \frac1C \cdot r^d \leq \beta_{G,S}(r) \leq C \cdot r^d.
    \]
    Consequently,
    \[ \limsup_{r\rightarrow \infty} 
       \frac{\beta_{G,S}(r-R)}{\beta_{G,S}(r)} = 1.
    \]
\end{prop}
\begin{proof}
  The finite set $T := \{g \in \Z^d \mid |g|_\infty = 1\}$ generates
  the additive monoid~$\Z^d$ and $d_T$ coincides with the
  $\infty$-metric. Now a simple counting argument shows the
  first part for~$\beta_{\Z^d,T}$. 
  For the finite monoid generating set~$T \cup F$ of~$G$ we clearly have
  \[ \beta_{\Z^d,T} \leq \beta_{G, T \cup F} \leq \beta_{\Z^d,T} \cdot |F|,
  \]
  so the first part holds also for~$\beta_{G,T\cup F}$.
  Proposition~\ref{propgenset} translates this into
  corresponding estimates for the monoid generating set~$S$ of~$G$.

  It follows from the first part in particular that the limes superior 
  in the second part indeed exists. Because $\beta_{G,S}$ is monotonically 
  increasing, the limes superior is at most~$1$; if the limes superior 
  were stricly less than~$1$, then $\beta_{G,S}$ would be growing exponentially, 
  contradicting the first part.
\end{proof}


\begin{prop}[Growth in product groups]\label{propproductgrowth}
  Let $G_1$ and $G_2$ be finitely generated groups, and let $S_1 \subset G_1$ 
  and $S_2 \subset G_2$ be finite monoid generating sets of~$G_1$ and $G_2$ 
  respectively. 
  Then $S := (S_1 \times \{e\}) \cup (\{e\} \cup S_2)$ is a finite monoid 
  generating set of~$G := G_1 \times G_2$ and 
  for all~$r \in \N$ we have
  \[ \sigma_{G, S}(r) 
     = \sum_{r_1 =0}^r \sum_{r_2 = 0}^{r-r_1} 
       \sigma_{G_1,S_1}(r_1) \cdot \sigma_{G_2, S_2}(r_2).
  \]
\end{prop}
\begin{proof}
  By definition of the word metric, for all~$(g_1, g_2) \in G$ we have
  \[ d_S\bigl((g_1, g_2), e\bigr) 
     = d_{S_1}(g_1, e) + d_{S_2}(g_2, e),
  \]
  which readily implies the stated decomposition of~$\sigma_{G,s}$.
\end{proof}

Moreover, we will use the following version of an observation by Hainke 
and Scheele~\cite{mathoverflow,loeh}:

\begin{prop}[Growth and elements of order~$2$]\label{propinv}
  Let $G$ be a finitely generated group and let $I \subset G$ be the
  set of elements of order at most~$2$. If $I$ is finite and $S \subset G$ 
  is a finite symmetric generating set of~$G$, then for all~$r \in \N$ 
  with $r > \diam_{d_s} I$ we have
  \[ \beta_{G,S}(r) \equiv |I| \mod 2. 
  \]
\end{prop}
\begin{proof}
  The inversion map~$i \colon G \longrightarrow G$ satisfies~$i
  (B_{G,S}(r)) \subset B_{G,S}(r)$ (and hence $i (B_{G,S}(r)) =
  B_{G,S}(r)$) for all~$r \in \N$ because $S$ is symmetric. Because $i
  \circ i = \id_G$ and because the fixed points of~$i$ are precisely the
  elements of order at most~$2$ the claim follows.
\end{proof}

\section{Necessary conditions}\label{secnec}

The conditions given in Theorem~\ref{thmsymmetric} and
Theorem~\ref{thmmonoid} are necessary:

\begin{prop}[Necessary conditions]
  Let $d, d' \in \N$, let $F, F'$ be finite groups, and let $G \cong
  \Z^d \times F$, $G' \cong \Z^{d'} \times F'$. Moreover, let $S
  \subset G$ and $S' \subset G'$ be finite monoid generating sets
  with~$\beta_{G,S} = \beta_{G',S'}$.
  \begin{enumerate}
    \item 
      Then $d = d'$.
    \item 
      If $S$ and $S'$ are symmetric, then $|F| \equiv |F'| \mod 2$.
  \end{enumerate}
\end{prop}
\begin{proof}
  It suffices to show that the rank and (in the symmetric case) the
  parity of the torsion part are encoded suitably in any growth
  function. The first part immediately follows from (the first part
  of) Proposition~\ref{proprankgrowth}. 

  The second part follows from Proposition~\ref{propinv}: On
  the one hand, applying Proposition~\ref{propinv} to~$F$ shows that
  the parity of~$|F|$ equals the parity of number of elements
  of order at most~$2$ in~$F$ (and thus in~$\Z^d \times F$). On the
  other hand, applying Proposition~\ref{propinv} to~$\Z^d \times F$
  shows that this number (and hence the parity of~$|F|$) is determined
  by the long-time behaviour of~$\beta_{G,S}$.
\end{proof}

\section{Sufficient conditions}\label{secsuff}

In view of the previous section, in order to prove
Theorems~\ref{thmsymmetric} and~\ref{thmmonoid} it remains to 
give examples of finite generating sets in the groups in question that 
witness that the corresponding groups admit equal growth functions.

\subsection{The symmetric case}

In the symmetric case, the following two examples will be at the heart of 
our arguments:

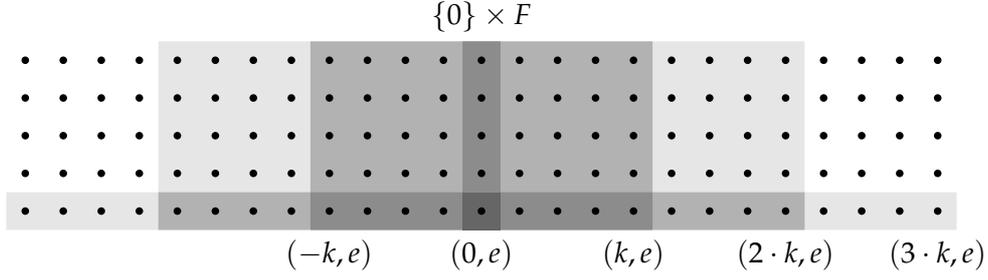
\begin{figure}
  \begin{center}
    \begin{tikzpicture}[x=0.5cm,y=0.5cm]
      \fill[black!10] (-12,0) -- (13,0) -- (13,1) 
                   -- (9,1) -- (9,5) -- (-8,5) -- (-8,1)
                   -- (-12,1) -- cycle;
      \fill[black!30] (-8,0) -- (9,0) -- (9,1) 
                   -- (5,1) -- (5,5) -- (-4,5) -- (-4,1)
                   -- (-8,1) -- cycle;
      \fill[black!45] (-4,0) -- (5,0) -- (5,1) 
                   -- (1,1) -- (1,5) -- (0,5) -- (0,1)
                   -- (-4,1) -- cycle;
      \fill[black!60] (0,0) rectangle +(1,1);
      \draw (0.5,5) node[anchor=south]{$\{0\} \times F$};
      \draw (0.5,0) node[anchor=north] {$(0,e)$};
      \draw (4.5,0) node[anchor=north] {$(k,e)$};
      \draw (8.5,0) node[anchor=north] {$(2\cdot k,e)$};
      \draw (12.5,0) node[anchor=north] {$(3 \cdot k,e)$};
      \draw (-3.5,0) node[anchor=north] {$(-k,e)$};
      \begin{scope}[shift={(0.5,0.5)}]
          \foreach \i in {-12,...,12}
          { \foreach \j in {0,...,4}
            { \gvertex{(\i,\j)}
            }
          }
      \end{scope}
    \end{tikzpicture}
  \end{center}

  \caption{Small balls in Example~\ref{exaeven}}
  \label{figeven}
\end{figure}

\begin{exa}\label{exaeven}
  Let $F$ be a finite group, let $G := \Z \times F$, let $k \in \N_{>0}$, and let 
  \[ S := \bigl( \{0\} \times F\bigr) 
          \cup
          \bigl( \{-k, \dots, k\} \times \{e\} \bigr)
          \subset G.
  \]
  Clearly, $S$ is a symmetric generating set of~$G$, and 
  Proposition~\ref{propproductgrowth} shows that
  \begin{align*}
    \sigma_{G,S} \colon \N & \longrightarrow \N \\
    r & \longmapsto 
    \begin{cases}
      1 & \text{if $r=0$} \\
      |F| - 1 + 2 \cdot k & \text{if $r=1$} \\
      (|F| - 1) \cdot 2 \cdot k + 2 \cdot k 
      = |F| \cdot 2 \cdot k & \text{if $r > 1$}
    \end{cases}
  \end{align*}
  (see also Figure~\ref{figeven}). 
  Notice that these terms are symmetric in~$|F|$ and $2 \cdot k$.
\end{exa}

\begin{exa}\label{exaodd}
  Let $F$ be a finite group, let $G := \Z \times F$, let $k \in \N_{>0}$, and let 
  \[ S := \bigl( \{0\} \times F\bigr) 
          \cup
          \bigl( \{ 2 \cdot j +1 \mid j \in \{-k, \dots, k-1\} \} \times \{e\} \bigr)
          \subset G.
  \]
  Clearly, $S$ is a symmetric generating set of~$G$, and 
  Proposition~\ref{propproductgrowth} shows that 
  \begin{align*}
    \sigma_{G,S} \colon \N \longrightarrow & \N \\
    r \longmapsto 
    &  
    \begin{cases}
      1 & \text{if $r=0$} \\
      |F| - 1 + 2 \cdot k & \text{if $r=1$} \\
      (|F| - 1) \cdot 2 \cdot k + 2 \cdot (k-1) + 2 \cdot k  
      & \text{if $r=2$} \\
        ( |F| - 1) \cdot 2 \cdot  (k -1) 
      + ( |F| - 1) \cdot 2 \cdot  k 
      + 2 \cdot (k-1) + 2 \cdot  k
      & \text{if $r > 2$}.   
    \end{cases}\\
    = & 
    \begin{cases}
      1 & \text{if $r=0$} \\
      |F| + 2 \cdot k - 1 & \text{if $r=1$} \\
      |F| \cdot (2 \cdot k -1) + |F| + 2 \cdot k -1  -1& \text{if $r=2$} \\
      2 \cdot |F| \cdot (2 \cdot k -1) 
      & \text{if $r > 2$}
    \end{cases}
  \end{align*}
  (see also Figure~\ref{figodd}). 
  Notice that these terms are symmetric in~$|F|$ and~$2 \cdot k -1$.
\end{exa}

\begin{figure}
  \begin{center}
    \makebox[0pt]{%
    \begin{tikzpicture}[x=0.5cm,y=0.5cm]
      \fill[black!10] (-15,0) rectangle +(1,1);
      \fill[black!10] (-13,0) rectangle +(1,1);
      \fill[black!10] (-11,0) rectangle +(1,1);
      \fill[black!10] ( 11,0) rectangle +(1,1);
      \fill[black!10] ( 13,0) rectangle +(1,1);
      \fill[black!10] ( 15,0) rectangle +(1,1);
      \fill[black!10] (-10,1) rectangle +(1,4);
      \fill[black!10] ( -8,1) rectangle +(1,4);
      \fill[black!10] ( -6,1) rectangle +(1,4);
      \fill[black!10] (  6,1) rectangle +(1,4);
      \fill[black!10] (  8,1) rectangle +(1,4);
      \fill[black!10] ( 10,1) rectangle +(1,4);
      \fill[black!10] ( -10,0) rectangle +(21,1);
      \fill[black!10] ( -5,1) rectangle +(11,4);
      \fill[black!30] (-10,0) rectangle +(1,1);
      \fill[black!30] ( -8,0) rectangle +(1,1);
      \fill[black!30] ( -6,0) rectangle +(1,1);
      \fill[black!30] (  6,0) rectangle +(1,1);
      \fill[black!30] (  8,0) rectangle +(1,1);
      \fill[black!30] ( 10,0) rectangle +(1,1);
      \fill[black!30] ( -5,1) rectangle +(1,4);
      \fill[black!30] (- 3,1) rectangle +(1,4);
      \fill[black!30] (- 1,1) rectangle +(1,4);
      \fill[black!30] (  1,1) rectangle +(1,4);
      \fill[black!30] (  3,1) rectangle +(1,4);
      \fill[black!30] (  5,1) rectangle +(1,4);
      \fill[black!30] ( -5,0) rectangle +(11,1);
      \fill[black!45] ( -5,0) rectangle +(1,1);
      \fill[black!45] (- 3,0) rectangle +(1,1);
      \fill[black!45] (- 1,0) rectangle +(1,1);
      \fill[black!45] (  1,0) rectangle +(1,1);
      \fill[black!45] (  3,0) rectangle +(1,1);
      \fill[black!45] (  5,0) rectangle +(1,1);
      \fill[black!45] (  0,1) rectangle +(1,4);
      \fill[black!60] (0,0) rectangle +(1,1);
      \draw (0.5,5) node[anchor=south]{$\{0\} \times F$};
      \draw (0.5,0) node[anchor=north] {$(0,e)$};
      \draw (5.5,0) node[anchor=north] {$(2\cdot k-1,e)$};
      \draw (-4.5,0) node[anchor=north] {$(-(2\cdot k-1),e)$};
      \begin{scope}[shift={(0.5,0.5)}]
          \foreach \i in {-15,...,15}
          { \foreach \j in {0,...,4}
            { \gvertex{(\i,\j)}
            }
          }
      \end{scope}
    \end{tikzpicture}}
  \end{center}

  \caption{Small balls in Example~\ref{exaodd}}
  \label{figodd}
\end{figure}
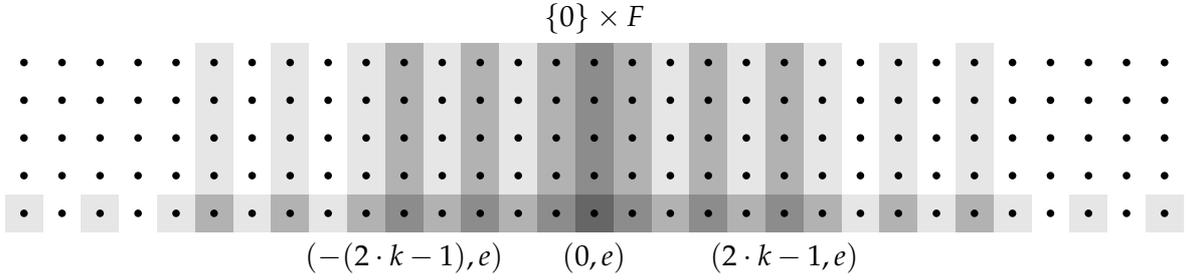

\begin{prop}[Witnesses in the symmetric case]
  Let $d \in \N_{>0}$, and let $F_1$ and $F_2$ be finite groups. If $|F_1|$ 
  and $|F_2|$ have the same parity, then there exist finite symmetric generating 
  sets~$S_1$ and $S_2$ of~$G_1 := \Z^d \times F_1$ and $G_2 := \Z^d \times F_2$ 
  respectively satisfying
  $\sigma_{G_1, S_1} = \sigma_{G_2, S_2}
  $
  and thus also
  \[\beta_{G_1,S_1} = \beta_{G_2,S_2}. 
  \]
\end{prop}

\begin{proof}
  It suffices to consider the case~$d=1$: If~$d >1$, we can just
  extend symmetric generating sets for~$\Z \times F_1$ and $\Z\times
  F_2$ that witness that $\Z \times F_1$ and $\Z \times F_2$ admit
  equal growth functions by a finite symmetric generating set
  of~$\Z^{d-1}$ and apply Proposition~\ref{propproductgrowth} to
  produce finite symmetric generating sets for~$G_1$ and $G_2$ that
  witness that $G_1 = \Z^{d-1} \times \Z \times F_1$ and $G_2 =
  \Z^{d-1} \times \Z \times F_2$ admit equal growth functions.

  We begin with the \emph{even case}: So, let $|F_1|$ and $|F_2|$ 
  be even, say~$|F_1| = 2 \cdot k_1$ and $|F_2| = 2\cdot k_2$ for
  certain~$k_1, k_2 \in \N_{>0}$. Then
  \begin{align*}
    S_1 & := \bigl( \{0\} \times F_1\bigr) 
          \cup
          \bigl( \{-k_2, \dots, k_2\} \times \{e\} \bigr)
          \subset \Z \times F_1, \\
    S_2 & := \bigl( \{0\} \times F_2 \bigr) 
          \cup
          \bigl( \{-k_1, \dots, k_1\} \times \{e\} \bigr)
          \subset \Z \times F_2
  \end{align*}
  are finite symmetric generating sets of~$\Z \times F_1$ and $\Z \times F_2$ 
  respectively, and Example~\ref{exaeven} shows that
  $   \sigma_{\Z \times F_1, S_1}
     = \sigma_{\Z \times F_2, S_2}
    ,
  $
  as desired.

  It remains to deal with the \emph{odd case}: So, let $|F_1|$ and $|F_2|$ 
  be odd, say~$|F_1| = 2 \cdot k_1 -1$ and $|F_2| = 2\cdot k_2-1$ for
  certain~$k_1, k_2 \in \N_{>0}$. Then
  \begin{align*}
    S_1 & := 
\bigl( \{0\} \times F_1 \bigr) 
          \cup
          \bigl( \{ 2 \cdot j +1 \mid 
          j \in \{-k_2, \dots, k_2-1\} \} \times \{e\} \bigr)
          \subset \Z \times F_1    \\
    S_2 & := 
\bigl( \{0\} \times F_2 \bigr) 
          \cup
          \bigl( \{ 2 \cdot j +1 \mid 
          j \in \{-k_1, \dots, k_1-1\} \} \times \{e\} \bigr)
          \subset \Z \times F_2
  \end{align*}
  are finite symmetric generating sets of~$\Z \times F_1$ and $\Z \times F_2$ 
  respectively, and Example~\ref{exaodd} shows that
  $   \sigma_{\Z \times F_1, S_1}
     = \sigma_{\Z \times F_2, S_2}
    ,
  $
  as desired.
\end{proof}

This finishes the proof of Theorem~\ref{thmsymmetric}.

\begin{rem}
  The construction in Example~\ref{exaeven} and~\ref{exaodd} does 
  not produce witnesses for other non-trivial cases: The system 
  \begin{align}
    x_1 + y_1     & = x_2 + y_2, \label{eq1} \\
    x_1 \cdot y_1 & = x_2 \cdot y_2 \label{eq2}
  \end{align}
  (corresponding to the constraints for radius~$1$ and larger radii, respectively)
  with~$x_1, x_2, y_1, y_2 \in \N$ has only the two solutions 
  where 
  $x_1 = x_2, y_1 = y_2$ or $x_1 = y_1, x_2 = y_2$.

  This can be easily seen as follows: Solving Equation~\eqref{eq1} 
  for~$x_2$ and using Equation~\eqref{eq2} yields
  \begin{align*}
        (x_1 - y_2) \cdot (y_1 -y_2) 
    & = x_1 \cdot y_1 - y_2 \cdot (y_1 + x_1 - y_2) \\
    & = x_1 \cdot y_1 - y_2 \cdot x_2 \\
    & = 0,
  \end{align*}
  which implies~$x_1 = y_2$ (and hence~$x_2 = y_2$) 
  or $y_1 = y_2$ (and hence~$x_1 = x_2$). 
\end{rem}

\subsection{The monoid case}

Similarly to the symmetric case, we start with the corresponding 
key example:

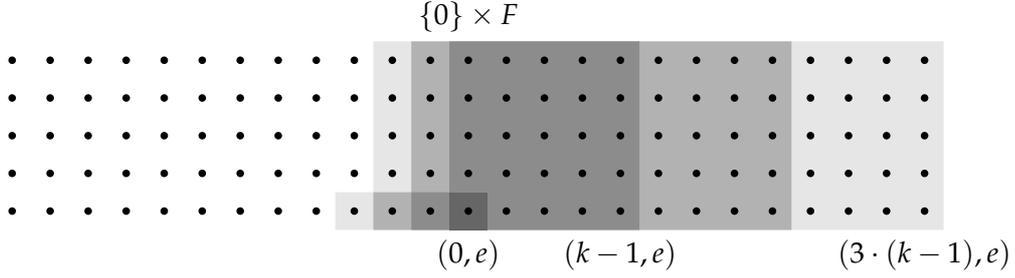
\begin{figure}
  \begin{center}
    \begin{tikzpicture}[x=0.5cm,y=0.5cm]
      \fill[black!10] (-3,0) -- (13,0) -- (13,5) 
                   -- (-2,5) -- (-2,1) -- (-3,1) -- cycle;
      \fill[black!30] (-2,0) -- (9,0) -- (9,5) 
                   -- (-1,5) -- (-1,1) -- (-2,1) -- cycle;
      \fill[black!45] (-1,0) -- (5,0) -- (5,5) 
                   -- (0,5) -- (0,1) -- (-1,1) -- cycle;
      \fill[black!60] (0,0) rectangle +(1,1);
      \draw (0.5,5) node[anchor=south]{$\{0\} \times F$};
      \draw (0.5,0) node[anchor=north] {$(0,e)$};
      \draw (4.5,0) node[anchor=north] {$(k-1,e)$};
      \draw (12.5,0) node[anchor=north] {$(3 \cdot (k-1),e)$};
      \begin{scope}[shift={(0.5,0.5)}]
          \foreach \i in {-12,...,12}
          { \foreach \j in {0,...,4}
            { \gvertex{(\i,\j)}
            }
          }
      \end{scope}
    \end{tikzpicture}
  \end{center}
  \caption{Small balls in Example~\ref{examonoid}}
  \label{figmonoid}
\end{figure}

\begin{exa}[\cite{mann}]\label{examonoid}
  Let $F$ be a finite group, let $G := \Z \times F$, let $k\in \N_{>0}$, 
  and let 
  \[ S := \bigl\{(-1,e)\bigr\} \cup \bigl(\{0,\dots,k-1\} \times F\bigr).
  \]
  Then $S$ is a finite monoid generating set of~$G$ (which except for trivial
  cases is not symmetric), and a straightforward induction over the 
  radius of balls shows that 
  \begin{align*}
    B_{G,S}(r) = \bigl\{ (-r,e)\bigr\} 
                \cup \bigl(\{-(r-1), \dots, r \cdot (k-1)\} \times F\bigr)
  \end{align*}
  for all~$r \in \N_{>0}$. Hence, we obtain (see also Figure~\ref{figmonoid})
  \begin{align*}
    \sigma_{G,S} \colon \N & \longrightarrow \N \\
    r & \longmapsto
    \begin{cases}
      1 & \text{if $r=0$}\\
      |F| \cdot (k-1) + |F| -1 + 1  = |F| \cdot k& \text{if $r>0$}.
    \end{cases}
  \end{align*}
\end{exa}

\begin{prop}[Witnesses in the monoid case \cite{mann}]
  Let $m \in \N_{>0}$ and let $F_1, \dots, F_m$ be finite groups. Moreover,
  let $d \in \N_{>0}$ and let $G_j := \Z^d \times F_j$ for all~$j \in
  \{1,\dots,m\}$.  Then for every~$j \in \{1,\dots,m\}$ there exists 
  a finite monoid generating set~$S_j \subset G_j$ of~$G_j$ such that
  $\sigma_{G_1,S_1} = \dots = \sigma_{G_m,S_m}
  $
  and thus also
  \[ \beta_{G_1,S_1} = \dots = \beta_{G_m,S_m}.
  \]
\end{prop}
\begin{proof}
  As in the symmetric case, in view of
  Proposition~\ref{propproductgrowth} it suffices to consider the
  case~$d=1$. So, let $d = 1$ and let $K \in \N_{>0}$ be some common 
  multiple of~$|F_1|, \dots, |F_m|$. For~$j \in \{1,\dots,m\}$ we 
  then consider the finite monoid generating set
  \[ S_j := \{(-1,e)\} \cup \bigl(\{0,\dots, K/|F_j|-1\} \times F_j\bigr) 
  \]
  of~$G_j$. By Example~\ref{examonoid} we have $\sigma_{G_j, S_j}(0)=1$ 
  and 
  \[ \sigma_{G_j, S_j}(r) 
     = |F_j| \cdot K /|F_j| = K
  \]
  for all~$r \in \N_{>0}$, which is independent of~$j$. Hence, 
  $ \sigma_{G_1,S_1} = \dots = \sigma_{G_m,S_m},
  $
  as desired.
\end{proof}

This completes the proof of Theorem~\ref{thmmonoid}.

\section{Minimal growth of finitely generated Abelian groups}\label{secmin}

In the following, we prove Proposition~\ref{propmingrowth}
and its consequences from the introduction.

\subsection{Minimal growth}

We start with the proof of the following version of Proposition~\ref{propmingrowth}:

\begin{prop}[Minimal growth]\label{propmingrowthgen}
  Let $G \cong \Z^d \times F$, where $d \in \N$ and $F$ is a finite group.
  \begin{enumerate}
    \item If $S$ is a finite symmetric generating set of~$G$, then
      \[ \limsup_{r \rightarrow \infty} 
      \frac{\beta_{G,S}(r)}{\beta_{d}(r)} 
      \geq |F|.
      \]
    \item If $S$ is a finite monoid generating set of~$G$, then 
      \[ \limsup_{r \rightarrow \infty} 
      \frac{\beta_{G,S}(r)}{\beta^+_{d}(r)} 
      \geq |F|.
      \]
  \end{enumerate}
\end{prop}

We first reduce to the free Abelian case (Lemma~\ref{lemredtofree}), 
and then compare growth functions in the free Abelian case with the 
standard growth functions (Lemma~\ref{lemfree}).

\begin{lem}[Reduction to the free Abelian case]\label{lemredtofree}
  Let $G \cong \Z^d \times F$, where $d \in \N$ and $F$ is a finite
  group, let $S \subset G$ be a finite monoid generating set of~$G$,
  let $\pi \colon G \cong \Z^d \times F \longrightarrow \Z^d$ be the
  canonical projection, and let $R := \diam_{d_S} F$. Then
  \[ \beta_{G,S}(r) 
     \geq 
     |F| \cdot \beta_{\Z^d, \pi(S)}(r-R)
  \]
  for all~$r \in \N_{\geq R}$.
\end{lem}
\begin{proof}
  Because $\pi$ is surjective, $\pi(S)$ is indeed a generating set
  of~$\Z^d$; moreover, $F$ is finite, and so $R$ is
  well-defined.  Let $r \in \N_{\geq R}$. Then, by definition of 
  the word metric, 
  \[ \pi \bigl( B_{G,S}(r-R)\bigr) = B_{\pi(G),\pi(S)}(r-R)  
  \]
  and, by definition of~$R$, we have 
  \[ B_{G,S}(r) \supset B_{G,S}(r - R) + F. 
  \]
  Hence, we obtain
  \begin{align*}
    \beta_{G,S}(r)  
    & \geq \bigl|B_{G,S}(r - R) + F\bigr|
    = \bigl| \pi^{-1}\bigl(\pi(B_{G,S}(r-R)\bigr)\bigr|
    \\
    & = |F| \cdot \bigl| \pi \bigl( B_{G,S}(r-R)\bigr)\bigr|  
      = |F| \cdot \bigl|B_{\pi(G),\pi(S)}(r-R)\bigr|
    \\
    & = |F| \cdot \beta_{\Z^d, \pi(S)}(r-R).
    \qedhere
  \end{align*}
\end{proof}

\begin{lem}[Minimal growth of free Abelian groups]\label{lemfree}
  Let $d \in \N$.
  \begin{enumerate}
    \item Let $S \subset \Z^d$ be a finite symmetric generating set. Then
      \[ \beta_{\Z^d,S} \geq \beta_d. 
      \]
    \item Let $S \subset \Z^d$ be a finite monoid generating set. Then
      \[ \beta_{\Z^d,S} \geq \beta^+_d. 
      \]
  \end{enumerate}
\end{lem}
\begin{proof}
  Let $S \subset \Z^d$ be a finite monoid generating set. Looking
  at~$\Z^d \otimes \Q$ shows that $S$ contains a $d$-element
  subset~$E$ that is linearly independent over~$\Q$. In particular,
  the submonoid~$N$ of~$\Z^d$ generated by~$E$ is isomorphic to~$\N^d$,
  and the subgroup~$Z$ of~$\Z^d$ generated by~$E$ is isomorphic to~$\Z^d$;
  in both cases, $E$ is a free generating set of the corresponding
  submonoid or subgroup, respectively.

  We now prove the first part of the lemma: If $S$ is symmetric, 
  then also $-E \subset S$, and we obtain
  \[ \beta_{\Z^d, S} \geq \beta_{Z, E \cup (-E)} = \beta_{\Z^d, E_d \cup (-E_d)} 
     =\beta_d.
  \]

  For the second part, the combinatorics is slightly more complicated 
  because not every finite monoid generating set of~$\Z^d$ contains a 
  generating set corresponding to~$E_d \cup \{v_d\}$. 

  In order to prove the second part, it suffices to construct an injective
  map~$\varphi \colon \Z^d \longrightarrow \Z^d$ that maps
  $E_d\cup\{v_d\}$-balls into $S$-balls of the same radius. We will 
  now give the construction of such a map:

  We choose an order~$(e'_1, \dots, e'_d)$ on~$E$ and denote
  by~$\pi_1, \dots, \pi_d \colon \Q^d \longrightarrow \Q$ the
  coordinate maps corresponding to the (ordered) basis~$E$ of~$\Z^d
  \otimes \Q$. Because $S$ is a monoid generating set of~$\Z^d$, for
  each~$j \in \{1,\dots, d\}$ there exists~$v'_j \in S$
  with~$\pi_j(v'_j) < 0$; we choose~$v'_j$ in such a way that it
  minimises~$\pi_j$ on~$S$. We denote the set of minimal
  $E_d$-coordinates of an element~$x \in \Z^d$ by
  \[ M(x) := \bigl\{ j \in \{1,\dots,d\} \bigm| x_j = \min_{k \in\{1,\dots,d\}} x_k \bigr\}, 
  \]
  and the set of minimal rescaled $E$-coordinates of an element~$x \in \Q^d$ 
  by
  \[ M'(x) := \Bigl\{ j \in \{1,\dots,d\} 
             \Bigm| 
             \frac{\pi_j(x)}{|\pi_j(v'_j)|} = \min_{k \in\{1,\dots,d\}} \frac{\pi_k(x)}{|\pi_k(v'_k)|} 
             \Bigr\}; 
  \]
  moreover, in this situation, we write
  \[ m(x) := \min M(x)
     \quad
     \text{and}
     \quad
     m'(x) := \min M'(x),
  \]
  respectively.
  
  We now define the map~$\varphi \colon \Z^d \longrightarrow \Z^d$ as
  follows: Let $x \in \Z^d$. Then one easily sees that $x$ has a unique
  minimal representation 
  \[ x = \sum_{j=1}^d x'_j \cdot e_j + x' \cdot v_d 
  \]
  with $x'_1, \dots, x'_d, x' \in \N$ with respect to the word
  metric~$d_{E_d \cup \{v_d\}}$. We set
  \[ \varphi(x) 
     := 
     \sum_{j=1}^d x'_j \cdot e'_j + x' \cdot v'_{m(x)}.
  \]

  By construction, we have $\varphi(B_{\Z^d, E_d \cup \{v_d\}}(r))
    \subset B_{\Z^d, S}(r)$ for all~$r \in \N$. It remains to show that $\varphi$ is injective: 
  Clearly, $\varphi|_{\N^d}$ is injective, and, by construction, $\pi_j(\varphi(x))\geq 0$ 
  for all~$j \in \{1,\dots, d\}$ if and only if $x \in \N^d$. 

  In case $x \in \Z^d \setminus \N^d$ (which is equivalent to~$x' >
  0$) we have
  \[  m'\bigl(\varphi(x)\bigr) = m(x).
  \]
  Hence, we can reconstruct $x'$ from the value~$\varphi(x)$ as the
  minimal natural number~$a$ such that all $E$-coordinates
  of~$\varphi(x) - a \cdot v'_{m(x)}$ are non-negative; because $E$ is
  free, we can then also read off~$x'_1, \dots, x'_d$
  from~$\varphi(x)$. Thus, $x$ is determined uniquely by~$\varphi(x)$,
  and so $\varphi$ is injective. This finishes the proof of
  Lemma~\ref{lemfree}.
\end{proof}

We can now combine these two steps to complete the proof of 
Proposition~\ref{propmingrowthgen}:

\begin{proof}[Proof of Proposition~\ref{propmingrowthgen}]
  Let $R := \diam_{d_S} \tors G$. 
  We begin with the \emph{symmetric case}: In view of Lemma~\ref{lemredtofree} 
  and~\ref{lemfree}, we have 
  $\beta_{G,S}(r) \geq |F| \cdot \beta_d(r-R)
  $
  for all~$r \in \N_{\geq R}$. Therefore, 
  \begin{align*}
    \limsup_{r \rightarrow \infty} 
    \frac{\beta_{G,S}(r)}{\beta_{d}(r)}
    &\geq
    |F| 
    \cdot
    \limsup_{r \rightarrow \infty} 
    \frac{\beta_{d}(r-R)}{\beta_{d}(r)}
    .
  \end{align*}
  The limes superior on the right hand side is equal to~$1$ by
  Proposition~\ref{propliminf}, which gives the desired estimate.

  Using the corresponding cases for monoid generating sets 
  allows to prove the \emph{monoid case} by the same arguments.
\end{proof}

\subsection{Consequences of minimal growth: Finite ambiguity}

We now prove the finiteness statement Corollary~\ref{corfinite}: 

\begin{cor}[Finite ambiguity]
  Let $\beta \colon \N \longrightarrow \N$ be a function.
   Then there is at most one~$d\in \N$ and at most finitely many
   isomorphism types of finite groups~$F$ such that $\Z^d \times F$
   has a finite monoid generating set~$S$ with~$\beta_{\Z^d \times
     F,S} = \beta$.
\end{cor}

\begin{proof}
  Suppose $d \in \N$ and $F$ is a finite group such that there exists
  a finite monoid generating set~$S$ of~$G := \Z^d \times F$
  with~$\beta_{G,S} = \beta$. Then $d$ is determined by the growth 
  rate of~$\beta$ (Proposition~\ref{proprankgrowth}), the limes superior
  \[ \limsup_{r \rightarrow \infty} 
         \frac{\beta(r)}{\beta^+_d(r)}
    =\limsup_{r \rightarrow \infty} 
         \frac{\beta_{G,S}(r)}{\beta^+_d(r)}
  \]
  is finite (Proposition~\ref{propgenset} and~\ref{proprankgrowth}), 
  and by Proposition~\ref{propmingrowthgen} we 
  have
  \begin{align*}
    |F| 
    \leq \limsup_{r \rightarrow \infty} 
         \frac{\beta_{G,S}(r)}{\beta^+_d(r)}
    = \limsup_{r \rightarrow \infty} 
         \frac{\beta(r)}{\beta^+_d(r)}
   < \infty.
  \end{align*}
  Hence, $|F|$ is bounded in terms of~$\beta$. As there are only
  finitely many isomorphism types of groups of a given finite order
  there are only finitely many different candidates of isomorphism
  types for~$F$.
\end{proof}

\subsection{Consequences of minimal growth: Recognising the size of torsion from growth sets}

In this section, we show that the set of all growth functions encodes
the size of the torsion part (Corollary~\ref{corgrowthset}):

\begin{cor}[Recognising the size of the torsion part from the set 
  of growth functions]
  Let $d,d' \in \N$, let $F$ and $F'$ be finite groups, and let
  $G \cong \Z^d \times F$ as well as $G' \cong \Z^{d'} \times F'$.
  \begin{enumerate}
  \item
    If
      \begin{align*} 
         & \; \{ \beta_{G,S} 
              \mid \text{$S$ is a finite symmetric generating set of~$G$}\}
         \\
       = & \; \{ \beta_{G',S'} 
              \mid \text{$S'$ is a finite symmetric generating set of~$G'$}\},
      \end{align*}
      then $d=d'$ and $|F| = |F'|$.
  \item
    If
      \begin{align*} 
         & \; \{ \beta_{G,S} 
              \mid \text{$S$ is a finite monoid generating set of~$G$}\}
         \\
       = & \; \{ \beta_{G',S'} 
              \mid \text{$S'$ is a finite monoid generating set of~$G'$}\},
      \end{align*}
      then $d = d'$ and $|F| = |F'|$.
  \end{enumerate}
\end{cor}

\begin{proof}
  In view of Proposition~\ref{proprankgrowth} it suffices to show that
  the sizes of the torsion parts must be equal if the sets of growth
  functions coincide.

  We begin with the \emph{symmetric case}: We consider the finite 
  symmetric generating set (see Definition~\ref{defstandardgrowth} 
  for the definition of~$E_d$)
  \[ S := E_d \cup (-E_d) \cup F
  \]
  of~$G$ (viewing $\Z^{\rk G}$ and $F$ as subsets of~$G$). For all~$r
  \in \N$ we have
  \begin{align*}
    \beta_{G,S}(r) & = \bigl| B_{G,S} (r)\bigr| 
                  \leq \bigl| B_{\Z^d,E_d \cup (-E_d)}(r) + F \bigr| = \beta_{d}(r) \cdot |F|,
  \end{align*}
  and hence
  \begin{align*}
    \limsup_{r\rightarrow \infty} 
      \frac{\beta_{G,S}(r)}
           {\beta_{d}(r)} 
    & \leq
      \limsup_{r\rightarrow \infty}
      \frac{\beta_{d}(r)\cdot |F|}
           {\beta_{d}(r)} 
     = |F|.
  \end{align*}
  Because we assumed that the growth sets of~$G$ and~$G'$ coincide, $\beta_{G,S}$ 
  is also a growth function of~$G'$ with respect to some finite symmetric 
  generating set of~$G'$ and $d = d'$. 
  In combination with Proposition~\ref{propmingrowthgen} we therefore obtain
  \[ |F'| 
     \leq \limsup_{r\rightarrow \infty} 
          \frac{\beta_{G,S}(r)}{\beta_{d}(r)}           
                                        \\
     \leq |F|.
  \]
  Hence, by symmetry of the setup, $|F'| = |F|$, as claimed.

  In the \emph{monoid case} the same argument applies with respect to 
  the monoid generating set
  $E_d \cup \{v_d\} \cup F
  $
  of~$\Z^{d} \times F \cong G$.
\end{proof}

However, the converse does not hold in general:

\begin{exa}\label{exaconverse}
  Let $d \in \N$ and let $F$ be an Abelian finite group that cannot be
  generated by $d+2$~elements, e.g., $F = (\Z/2)^{d + 3}$.  We then
  consider the groups
  \begin{align*}
    G  & := \Z^d \times \Z/|F|
    \quad\text{and}\quad
    G' := \Z^d \times F.
  \end{align*}
  Clearly, $S := (E_d  \times \{0\}) \cup \{(0,1), (0,-1)\}$ is a 
  finite symmetric generating set of~$G$. So, 
  $\beta_{G,S}(1) = \bigl|S \cup \{0,e\}\bigr| = d + 3. 
  $
  On the other hand, if $S' \subset G'$ is a finite monoid generating 
  set of~$G'$, then the canonical projection of~$S'$ to~$F$ must 
  generate~$F$ as a group; hence, $|S'| > d+2$, and so
  \[ \beta_{G',S'}(1) = \bigl|S'\cup \{0,e\} \bigr| > d+3 = \beta_{G,S}(1). 
  \]
  In particular, $\beta_{G,S}$ cannot be realised as a growth function 
  of~$G'$. However, by construction, $\rk G = d = \rk G'$ and
  $|\tors G| = \bigl|\Z/|F|\bigr| = |F| = |\tors G'|$.
\end{exa}


\medskip
\vfill

\noindent
\emph{Clara L\"oh}, 
\textsf{clara.loeh@mathematik.uni-r.de},
\textsf{www.mathematik.uni-r.de/loeh}
\\
\emph{Matthias Mann},
\textsf{matthias\_mann@gmx.de}

\smallskip

  {\small
  \begin{tabular}{@{\quad}l}
    Fakult\"at f\"ur Mathematik\\
    Universit\"at Regensburg\\
    93040 Regensburg\\
    Germany\\
  \end{tabular}}
\end{document}